\documentclass[reqno]{amsart}

\usepackage{verbatim}
\usepackage[textsize=scriptsize]{todonotes}
\usepackage{tikz-cd}
\usepackage{etex}
\usepackage{hyperref}
\usepackage{cleveref}
\usepackage[T1]{fontenc}

\usepackage[all,cmtip]{xy}
\usepackage{multicol}

\usepackage{longtable}

\hypersetup{
   colorlinks,
   linkcolor={red},
   citecolor={blue!50!black},
   urlcolor={blue!80!black}
}

\usepackage{chemarr}
\usepackage{amssymb}
\usepackage{amsmath}
\usepackage{comment}
\usepackage{spectralsequences}
\usepackage{cleveref}
\usepackage{mathtools}
\usepackage{rotating}
\usepackage{wrapfig}
\usepackage{outlines}
\usepackage{graphicx}
\usepackage{xcolor}
\usepackage{bbm}
\usepackage{enumitem}
\usepackage{mathdots}

\theoremstyle{definition}
\newtheorem{nul}{}[section]

\newtheorem{rmk}[nul]{Remark}

\newtheorem{cnstr}[nul]{Construction}

\newtheorem*{dfn*}{Definition}
\newtheorem*{axm*}{Axiom}
\newtheorem*{ntn*}{Notation}
\newtheorem*{exm*}{Example}
\newtheorem*{exr*}{Exercise}
\newtheorem*{int*}{Intuition}
\newtheorem*{qst*}{Question}
\newtheorem*{rmk*}{Remark}

\theoremstyle{plain}

\newtheorem{thm}[nul]{Theorem}
\newtheorem{prop}[nul]{Proposition}
\newtheorem{lem}[nul]{Lemma}

\newtheorem{cnj}[nul]{Conjecture}
\newtheorem{cor}[nul]{Corollary}

\newtheorem*{thm*}{Theorem}
\newtheorem*{prop*}{Proposition}
\newtheorem*{cor*}{Corollary}
\newtheorem*{lem*}{Lemma}
\newtheorem*{cnj*}{Conjecture}

\let\oldwidetilde\widetilde
\protected\def\widetilde{\oldwidetilde}

\DeclareMathOperator{\FF}{\mathbb{F}_2}

\DeclareMathOperator{\A}{\mathcal{A}}

\DeclareMathOperator{\bS}{\mathbb{S}}

\DeclareMathOperator{\Ext}{\mathrm{Ext}}

\DeclareMathOperator{\Sq}{\mathrm{Sq}}

\newcommand{\rp}{\mathbb{RP}}

\newcommand{\toto}{\Rightarrow}
\renewcommand{\P}{\mathcal P}
\newcommand{\fr}{\mathrm{fr}}

\def\P{\mathcal{P}}

\usepackage{tikz}
\usetikzlibrary{matrix,arrows,decorations}
\usepackage{tikz-cd}
\usepackage{float}

\usepackage{adjustbox}

\let\oldtocsection=\tocsection
 
\let\oldtocsubsection=\tocsubsection
 
\let\oldtocsubsubsection=\tocsubsubsection
 
\renewcommand{\tocsection}[2]{\hspace{0em}\oldtocsection{#1}{#2}}
\renewcommand{\tocsubsection}[2]{\hspace{1em}\oldtocsubsection{#1}{#2}}
\renewcommand{\tocsubsubsection}[2]{\hspace{2em}\oldtocsubsubsection{#1}{#2}}

\usepackage[margin=1in]{geometry}
\usepackage{etoolbox}
\newtoggle{draft}

\togglefalse{draft}

\iftoggle{draft} {
\newcommand{\NB}[1]{\todo[color=gray!40]{#1}}
\newcommand{\TODO}[1]{\todo[color=red]{#1}}
}{ 
\newcommand{\NB}[1]{}
\newcommand{\TODO}[1]{}
\renewcommand{\todo}[1]{}
\renewcommand{\todo}[1]{}
}

\usepackage{tikz}
\usetikzlibrary{tikzmark}
\usepackage{amssymb}
\usepackage{array}
\usepackage{rotating}

\title{The Adams differentials on the $e$-family}

\author{Runji Li}
\address{Uiversity of California, Los Angeles, Department of Mathematics, Los Angeles, CA 90095-1555, USA}
\email{runji@math.ucla.edu}

\author{Yuxuan Li}
\address{Qiuzhen College, Tsinghua University, Beijing 100084, China}
\email{yxli22@mails.tsinghua.edu.cn}

\begin{document}

\begin{abstract}
The New Doomsday Conjecture \cite{Minami} states that, for any nonzero $\Sq^0$-family, only finitely many terms in this family survive to the $E_\infty$-page. On the Adams $1$ and $2$-line, the conjecture, which corresponds to the Hopf invariant problem and the Kervaire invariant problem, were solved by Adams \cite{Adams} and Hill-Hopkins-Ravenel \cite{HHR}, respectively. On the Adams $3$-line, Burklund and Xu \cite{BXhj3} established a family of nontrivial differentials on the $h_j^3$ family, and in particular developed the Burklund-Xu Spectral Sequence, to study the non-triviality of its target on the Adams $E_2$-page. 

In this paper, we use the Burklund-Xu Spectral Sequence to establish the non-triviality of a product on the Adams $6$-line. Combining this with Bruner's formula \cite{Bruner,BMMS}, we prove the New Doomsday Conjecture for the $e$-family on the Adams $4$-line.
\end{abstract}

\maketitle

\setcounter{tocdepth}{1}
\tableofcontents
\vbadness 5000

\section{Introduction}
\label{sec:intro}

Computing the stable homotopy groups of spheres is a central problem in stable homotopy theory. Among the various methods developed for this purpose, the Adams spectral sequence stands out as one of the most powerful tools.
\[ E_2^{s,t} = \Ext_{\mathcal A}^{s,t}(\FF,\FF) \toto \pi_{t-s}(\bS_2) \]
Here $\A$ denotes the $\mathrm{mod\ }2$ Steenrod algebra, $s$ is the homological degree, $t$ is the internal degree, and $\bS_2$ denotes the $2$-completed sphere spectrum.

On the Adams $E_2$-page, we have algebraic Steenrod operations (See Chapter IV--VI of \cite{BMMS}) acting on the cohomology of the Steenrod algebra. In particular, there is an operation
\[ \Sq^0:\Ext_{\A}^{s,t}(\FF,\FF)\to \Ext_{\A}^{s,2t}(\FF,\FF) \]
Given an element $x\in \Ext_{\A}^{s,t}(\FF,\FF)$, we obtain an infinite $\Sq^0$-family of classes $\{x, \Sq^0(x), \Sq^0(\Sq^0(x)),\dots\}$ by iterating $\Sq^0$. We say that an $\Sq^0$-family is nontrivial if all members of the family are non-zero. Computations of the Adams spectral sequence lead to the following conjecture:

\begin{cnj}[New Doomsday Conjecture, \cite{Minami}] 
\label{conj:NDC}
    For any non-trivial $\Sq^0$-family $\{a_i\}_{i\ge 0}$ on the Adams $E_2$-page, only finitely many classes survive to the $E_\infty$-page.
\end{cnj}

Given an $Sq^0$-family in which only finitely many terms survive, we may further study how the terms die. Current computations lead to the following conjecture:

\begin{cnj}[Uniform Doomsday Conjecture, \cite{IWXsurveyICM,BXhj3}] 
\label{conj:UDC}
    For any non-trivial $Sq^0$-family $\{a_i\}_{i\ge 0}$ on the Adams $E_2$-page, there exists a non-trivial $Sq^0$-family $\{b_i\}_{i\ge 0}$, an integer $k\geq 2$ and a class $c$ on the Adams $E_2$-page such that $d_k(a_i) = c \cdot b_i$ for all but finitely many $i$.
\end{cnj}

On the Adams 1-line, the only classes are $h_i\in \Ext_{\A}^{1,2^i}(\FF,\FF), i\ge 0$ (see \cite{AdamsSseq}). The class $h_i$ survives if and only if there exists a map of Hopf invariant one at degree $2^i-1$. For this family, the conjecture was solved by Adams:
\begin{thm}[Adams, \cite{Adams}]\label{Hopfd2}
    $d_2(h_j) = h_0 h_{j-1}^2$, for $j\ge 4$.
\end{thm}

The remaining classes $h_i,0\le i\le 3$ are permanent cycles, which correspond to elements $2,\eta,\nu,\sigma$ in $\pi_*(\bS_2)$. These are the elements of Hopf invariant $1$. The corresponding manifolds can be taken as $S^0, S^1, S^3, S^7$, equipped with the real, complex, quaternionic and octonionic framing, respectively.

Again by Adams \cite{AdamsSseq}, the Adams $2$-line is generated by $h_ih_j$ under the relation $h_ih_{i+1}=0$. Most of the $\Sq^0$-families admit nontrivial differentials, deduced from Theorem~\ref{Hopfd2} and the Leibniz rule. The only  remaining family is $h_j^2$. The survival of $h_j^2$ is equivalent to the existence of a framed manifold with Kervaire invariant one of dimension $2^{j+1}-2$. The New Doomsday Conjecture on the Adams $2$-line was solved by Hill-Hopkins-Ravenel:

\begin{thm}[Hill-Hopkins-Ravenel, \cite{HHR}]
    The classes $h_j^2$ support nontrivial Adams differentials for $j\ge 7$.
\end{thm}

However, the explicit differentials supported by $h_j^2$ remain unknown. 

The remaining elements $h_i^2, 0\le i\le 6$ are permanent cycles. (\cite{BMT} for $i=4$, \cite{BJM62,Xu} for $i=5$, \cite{LWX} for $i=6$.) The corresponding manifolds are simply the squares of those corresponding to $h_i$, for $i\le 3$. When $i=4$, the explicit manifold is constructed by Jones in \cite{Jones}, and for larger $i$, there are no known manifolds corresponding to $h_i^2$.

The Adams $3$-line is generated by $c_j\in\Ext_{\A}^{3,2^j\cdot 11}(\FF,\FF),j\ge 0$ and products of $h_i$'s under the relations (see \cite{wangExt3})
    \[h_ih_{i+1}=0, \ h_ih_{i+2}^2=0,\ h_i^2h_{i+2} = h_{i+1}^3. \]
On this line, the New Doomsday Conjecture is still open. Two nontrivial cases $c_j, h_j^3$ were solved by Bruner and Burklund-Xu:

Let's recall Bruner's formula first: we define $v(n)=8a+2^b$, if $v_2(n+1) = 4a+b$ with $0\le b \le 3$. Then the attaching map of the $n$-cell of $\rp^n$ factors through $\rp^{n-v(n)}$, but not through $\rp^{n-v(n)-1}$.

If elements $c,d_1,d_2$ are in filtrations $s, s+r_1, s+r_2$ respectively, then
\[ d_*(c) = d_1 \dot + d_2 \]
means:

\begin{alignat*}{2}
  d_{r_1}(c) &= d_1       & \quad \text{if } & r_1 < r_2, \\
  d_r(c)     &= d_1 + d_2 & \quad \text{if } & r = r_1 = r_2, \\
  d_{r_2}(c) &= d_2       & \quad \text{if } & r_1 > r_2.
\end{alignat*}

\begin{thm}[Bruner's formula, \cite{Bruner,BMMS}]

    Let $x\in E_r^{s,t-s}$. Then
    \[
    d_* Sq^i x = Sq^{i+r-1} d_r x \dot+ 
    \begin{cases} 
      0 & v > k + 1 \text{ or } 2r - 2 < v < k \\
      a x d_r x & v = k + 1 \\
      a Sq^{i+v} x & v = k \text{ or } (v < k \text{ and } v \le 10)
    \end{cases}
    \]
    where $k = s-i, v = v(t-i)$, and $a\in E_\infty^{*,v-1}$ detects a generator of $\text{Im } J$ in $\pi_{v-1}\bS$. 
\end{thm}

Specializing Bruner's formula to $x = c_j$ and $i=0$, we get $d_2(c_j)= h_0f_{j-1}\neq 0$, for $j\ge 2$. (See \cite{LinExt,Ext5} for the non-triviality of the class $h_0 f_{j-1}$.) The remaining classes $c_0, c_1$ are permanent cycles.

\begin{thm}[Isaksen-Wang-Xu \cite{IWX0to90} for $j=5$, Burklund-Xu \cite{BXhj3} for $j\ge 6$]

    \[d_4(h_j^3) = h_0^3 g_{j-2}\text{ for all }j\ge 5\]
\end{thm}

Here $g_j\in\Ext_{\A}^{4, 2^{j+2}\cdot 3}(\FF,\FF), j\ge 1$ form an $\Sq^0$-family on the Adams $4$-line. The remaining elements $h_i^3$ are permanent cycles, for $i\le 4$. (See \cite{BMT})

The remaining open cases on the Adams $3$-line are $\{h_j^2h_{j+k+1} + h_{j+1}h_{j+k}^2\}_{j\ge 0}$, for $k\ge 2$.

\begin{rmk}
    The survival of $h_i$ and $h_i^2$ corresponds to the Hopf invariant problem and the Kervaire invariant problem, respectively. 
    More generally, the survival of an element of Adams filtration $n$ corresponds to the existence of a $(O,\fr)^n$-manifold (a generalization of manifolds with corners, with extra conditions on the stable normal bundle) with a certain characteristic number being nontrivial (See \cite{Laures,Miller}). From this perspective, the New Doomsday Conjecture implies that for any $\Sq^0$-family of characteristic numbers, $(O,\fr)^n$-manifolds detected by a member of this family exist in only finitely many dimensions. 
\end{rmk}

\begin{thm}[Main theorem]
\label{mainthm}
    On the Adams $E_2$-page, the element $h_0 x_j$ is nontrivial, for $j\ge 3$.
\end{thm}
Here $x_j\in \Ext_{\A}^{5,2^{j+1}\cdot 21}(\FF,\FF)$ is an $\Sq^0$-family on the Adams $5$-line (See \cite{Ext5}). The main theorem provides new data on the Adams $6$-line, which is beyond the current known range. For the remaining cases $j=0,1,2$, $h_0 x_j$ are also nontrivial, as verified using Lin's program \cite{LinGraph}. 

Specializing Bruner's formula gives a family of $d_2$-differentials:
\[ d_2(e_j) = h_0 x_{j-1},\ j\ge 2 \]
Combining with our result, we can show completely the fate of the $e_i$-family:

\begin{cor}
    For $j\ge 2$, we have the following nontrivial differentials:
    \[ d_2(e_j) = h_0 x_{j-1}\neq 0 \]
\end{cor}

This can be seen as one of the cases of the New Doomsday Conjecture (Conjecture~\ref{conj:NDC}) and the Uniform Doomsday Conjecture (Conjecture~\ref{conj:UDC}) on the Adams $4$-line.

\begin{rmk}

In \cite{Bruner}, Bruner determined the fate of the remaining elements in this family:
\[ d_2(e_0) = h_1^2d_0\neq 0,\ d_3(e_1) = h_1t\neq 0 \]

\end{rmk}

The strategy of our proof follows Chapter 4 of \cite{BXhj3}, where we employ the following two spectral sequences, and verify case-by-case that no differential kills $h_0 x_j$. The method has also been used to determine the $\tau^n$ torsion in the first 5 lines of the $E_2$-page of the $\mathbb{C}$-motivic Adams spectral sequence, see \cite{JordanBenson}.

\begin{cnstr}[See \cite{RavenelGreenBook,BXhj3}]
    Let $\mathcal P = \FF[\xi_1^2,\xi_2^2,\dots]$ be the sub-Hopf algebra of $\mathcal A$, and $\mathcal Q \cong \mathcal A \otimes_{\mathcal P}\FF$ be the quotient Hopf algebra. The extension of Hopf algebras $\mathcal P \to \mathcal A \to \mathcal Q$ gives rise to the Cartan-Eilenberg Spectral Sequence (CESS):
    \[ E_2^{s,k,t} = \Ext^{s,t}_{\mathcal P}(\FF,\Ext^k_{\mathcal Q}(\FF,\FF))\toto \Ext^{s+k,t}_{\mathcal A}(\FF,\FF) \]
    \[ d_r:E_r^{s,k,t}\to E_r^{s+r,k-r+1,t} \]
    where $s,k$ are homological degrees and $t$ is the internal degree.
    
    Furthermore, $\mathcal Q \cong \Lambda_{\FF}[\xi_1,\xi_2,\dots]$ is the cocommutative exterior algebra generated by primitive elements $\xi_1,\xi_2,\dots$. Therefore, the cohomology can be computed as:
    \[\Ext_{\mathcal Q}^{*,*}(\FF,\FF) \cong \FF[q_0,q_1,\cdots],\] where $q_i$ corresponds to $[\xi_{i+1}]$ and has degree $(k,t)=(1,2^{i+1}-1).$
    
    Moreover, the $\mathcal P$-comodule structure on $\Ext_{\mathcal Q}^{*,*}(\FF,\FF) \cong \FF[q_0,q_1,\cdots]$ is given by
    \[\psi(q_n) = \sum_{i=0}^n \xi_{n-i}^{2^{i+1}} \otimes q_i,\]
    where $\xi_0 = 1$.
\end{cnstr}

\begin{cnstr}
Following \cite{BXhj3}, we define a filtration on $\FF[q_0,q_1,\cdots]$ by assigning filtration $i$ to $q_i$. This filtration gives rise to the Burklund-Xu Spectral Sequence (BXSS):
    \[ E_1^{s,k,t,i} = \Ext^{s,t}_{\mathcal P}(\FF,\FF)\otimes \FF[q_0,q_1,\cdots]\toto \Ext^{s,t}_{\mathcal P}(\FF,\FF[q_0,q_1,\cdots])\]
    \[ d_r:E_r^{s,k,t,i}\to E_r^{s+1,k,t,i-r} \]
    where $q_i$ has degree $(0,1,2^{i+1}-1,i)$, and elements in $\Ext_{\mathcal P}^{s,t}(\FF,\FF)$ have degree $(s,k,t,i) = (s,0,t,0)$. Recall that the $(k,t)$-degree of $q_i$ is $(1,2^{i+1}-1)$.
\end{cnstr}

In summary, we have the following diagram
$$\xymatrix{
q_0 \cdot x_{n-1} \in \ar@{|->}[d] & \Ext_{\mathcal P}^{*,*}(\FF, \FF) \otimes \FF[q_0, q_1, \cdots] \ar@{=>}[d]^{\textup{Burklund-Xu Spectral Sequence}} \\
q_0 \cdot x_{n-1} \in \ar@{|->}[d] & \Ext^{*,*}_{\mathcal P}(\FF, \FF[q_0, q_1, \cdots]) \ar@{=>}[d]^{\textup{Cartan-Eilenberg Spectral Sequence}} \\
h_0 x_n \in & \Ext_{\mathcal A}^{*,*}(\FF, \FF)
}$$
To prove the main theorem, we will show that for $n\ge 1$, the element $q_0 \cdot x_{n-1}$ survives to the $E_\infty$-page of the BXSS in Section~\ref{sec:BXSS}, and it survives to the $E_\infty$-page of the CESS in Section~\ref{sec:CESS}.

\subsection*{Acknowledgments.}
The authors would like to thank Zhouli Xu for proposing this project, and Sihao Ma for helpful conversations.

\section{The Burklund-Xu Spectral Sequence}
\label{sec:BXSS}
The following proposition is a generalization of \cite[Lemma 4.10]{BXhj3}:

\begin{prop}[Differentials in BXSS] \label{lem diff algBXSS}
    Let $a\in \Ext_\P(\FF,\FF)$, so that $\langle h_{n+i}, h_{n+i+1}, \cdots, h_{n+k-1}, a \rangle$ contains $0$ in a compatible way for $i = 1,\dots, k-1$, then $d_i(q_{n+k}\cdot a)$ vanishes for $i\le k-1$, and
    \[ d_k(q_{n+k}\cdot a) = q_n\cdot \langle h_n,h_{n+1},\cdots,h_{n+k-1},a \rangle \]
\end{prop}

\begin{proof}
    By induction on $k$, we can assume that $d_i(q_{n+k}\cdot a)=0$, for $i = 1,2,\dots, k-1$.
    By abuse of notation, let $d$ be the cobar differential $\bar\P^{\otimes s}\otimes \FF[q_0,q_1,\cdots]\to \bar\P^{\otimes s+1}\otimes \FF[q_0,q_1,\cdots]$ as well as $\bar\P^{\otimes s}\to \bar\P^{\otimes s+1}$, where $\bar\P$ is the augmentation ideal, i.e. kernel of the map $\P\to \FF$. Then we have 
    \[ d(q_{n+k}\cdot a) = \sum_{i=0}^{n+k-1} a\otimes \xi_{n+k-i}^{2^{i+1}}\otimes q_i \]
    
    The condition of Toda bracket vanishing means that we can find classes $r_i \in \bar\P^{\otimes \mathrm{fil}(a)}$, so that
    \begin{align*}
        & d(r_1) = a\otimes \xi_1^{2^{n+k}}  \\
        & d(r_2) = a\otimes \xi_2^{2^{n+k-1}} + r_1\otimes \xi_1^{2^{n+k-1}}  \\
        & d(r_3) = a\otimes \xi_3^{2^{n+k-2}} + r_1\otimes \xi_2^{2^{n+k-2}} + r_2\otimes \xi_1^{2^{n+k-2}} \\
        & \vdots \\
        & d(r_{k-1}) = a\otimes \xi_{k-1}^{2^{n+2}} + r_1 \otimes \xi_{k-2}^{2^{n+2}} + \cdots + r_{k-2}\otimes \xi_1^{2^{n+2}} \\
    \end{align*}
    In other words, we have a diagram of Massey product:

    \[
        \begin{array}{*{13}{c}}
            a & & \xi_1^{2^{n+k}} & & \xi_1^{2^{n+k-1}} & & \cdots & & \xi_1^{2^{n+3}} & & \xi_1^{2^{n+2}} & & \xi_1^{2^{n+1}} \\
            & r_1 & & \xi_2^{2^{n+k-1}} & & \xi_2^{2^{n+k-2}} & & \cdots & & \xi_2^{2^{n+2}} & & \xi_2^{2^{n+1}} & \\
            & & r_2 & & \xi_3^{2^{n+k-2}} & & \xi_3^{2^{n+k-3}} & & \cdots & & \xi_3^{2^{n+1}} & & \\
            & & & r_3 & & \xi_4^{2^{n+k-3}} & & \ddots & & \iddots & & & \\
            & & & & \ddots & & \ddots & & \iddots & & & & \\
            & & & & & r_{k-1} & & \xi_k^{2^{n+1}} & & & & & 
        \end{array}
    \]

    Consider the triangular diagram below: 
    \begingroup
    \setlength{\arraycolsep}{1.5pt} 
    \footnotesize

    \[
    \setlength{\arraycolsep}{4pt} 
    \begin{array}{ccccccc}
        
          a\otimes \xi_1^{2^{n+k}}\otimes q_{n+k-1} 
        & a\otimes \xi_2^{2^{n+k-1}}\otimes q_{n+k-2} 
        & a\otimes \xi_3^{2^{n+k-2}}\otimes q_{n+k-3} 
        & a\otimes \xi_4^{2^{n+k-3}}\otimes q_{n+k-4} 
        & \cdots 
        & a\otimes \xi_{n+k}^2\otimes q_0 \\

        & r_1\otimes \xi_1^{2^{n+k-1}}\otimes q_{n+k-2} 
        & r_1\otimes \xi_2^{2^{n+k-2}}\otimes q_{n+k-3} 
        & r_1\otimes \xi_3^{2^{n+k-3}}\otimes q_{n+k-4} 
        & \cdots 
        & r_1\otimes \xi_{n+k-1}^{2}\otimes q_{0} \\

        & 
        & r_2\otimes \xi_1^{2^{n+k-2}}\otimes q_{n+k-3} 
        & r_2\otimes \xi_2^{2^{n+k-3}}\otimes q_{n+k-4} 
        & \cdots 
        & r_2\otimes \xi_{n+k-2}^{2}\otimes q_{0} \\

        & & & \ddots & & \vdots \\
        & & & & & 
    \end{array}
    \]
    \endgroup

    Here, the rows sum to $a\otimes d(q_{n+k}), r_1\otimes d(q_{n+k-1}),r_2\otimes d(q_{n+k-2}),\cdots$,\\
    and the columns sum to $d(r_1)\otimes q_{n+k-1}, d(r_2)\otimes q_{n+k-2}, \cdots$. So we can see that:

    \begin{align*}
        d(a\otimes q_{n+k}) + \sum_{i=1}^{k-1} d(r_i\otimes q_{n+k-i}) &= a\otimes d(q_{n+k}) + \sum_{i=1}^{k-1} r_i\otimes d(q_{n+k-i}) + \sum_{i=1}^{k-1} d(r_i)\otimes q_{n+k-i} \\
        &= (a \otimes \xi_k^{2^{n+1}} + \sum_{i=1}^{k-1}r_i\otimes \xi_{k-i}^{2^{n+1}})\otimes q_{n}
    \end{align*}
    
    The target corresponds to the element $q_n\cdot\langle h_n,h_{n+1},\cdots,h_{n+k-1}, a \rangle$ in the BXSS. Therefore, we have:
    \[ d_k(q_{n+k}\cdot a) = q_n\cdot\langle h_n,h_{n+1},\cdots,h_{n+k-1}, a \rangle \]
\end{proof}

Notice that, for degree reason, if $d_i(q_n\cdot a)=0$ for $i\le n$, then $q_n\cdot a$ is a permanent cycle.

\begin{cor}\label{lem permanent cycle BXSS}
    In the BXSS, classes of the form $q_n h_n, q_n h_{n+1}^2$ are permanent cycles for $n\geq 3$.
\end{cor}

\begin{proof}
    To apply \ref{lem diff algBXSS}, we need to show that these Toda brackets vanishes in a compatible way. We'll actually show a stronger statement: these Toda brackets vanish with no indeterminacy, for degree reason. 

    For $q_nh_n$, we need to verify that $\langle h_{n-k},h_{n-k+1},\cdots, h_n \rangle$ is $\{0\}$. This Toda bracket is in degree $(s,t) = (2,2^{n-k}+2^{n-k+1}+\cdots + 2^n)$. The corresponding $\Ext$ group is zero, as seen from the Adams $2$-line \cite{AdamsSseq}.

    For $q_nh_{n+1}^2$, we need to verify that $\langle h_{n-k},h_{n-k+1},\cdots, h_{n-1}, h_{n+1}^2 \rangle$ is $\{0\}$. This Toda bracket is in degree $(3,2^{n-k}+2^{n-k+1}+\cdots + 2^{n-1}+2^{n+2})$. The $3$-line of $\Ext_{\mathcal A}(\FF,\FF)$ is generated by (See \cite{wangExt3}) 
    \[ h_ih_jh_k \in \Ext_{\mathcal A}^{3,2^i+2^j+2^k}(\FF,\FF)\text{ and }c_i \in \Ext_{\mathcal A}^{3,2^{i+3}+2^{i+1}+2^i}(\FF,\FF)\] 
    under the relations
    \[ h_i h_{i+1}=0,\ h_ih_{i+2}^2=0,\ h_i^2h_{i+2} = h_{i+1}^3 \]  
    So the $\Ext$ group which contains our Toda bracket is $0$.
\end{proof}

Now let's go back to the explicit class $q_0 x_{n-1}$ in BXSS. The class $q_0 x_{n-1}$ is of degree \\
$(5,1,2^{n+1}\cdot 21+1, 0)$, so to show that it's not killed by any differentials, we need to list all the classes of degree
$(4,1,2^{n+1}\cdot 21+1,*)$.

\begin{lem}
    For $n\ge 3$, in degree $(4,1,2^{n+1}\cdot 21+1,*)$, the $E_1$-page of the BXSS consists of the following classes:
    \begin{enumerate}
        \item $q_0e_n$
        \item $q_{0}h_{n-1}^2h_{n+2}h_{n+4}$
        \item $q_{n-1}h_{0}h_{n-1}h_{n+2}h_{n+4}$
        \item $q_{n+2}h_{0}h_{n-1}^2h_{n+4}$
        \item $q_{n+4}h_{0}h_{n-1}^2h_{n+2}$
        \item $q_{n}h_{0}h_{n+1}^2h_{n+4}$
        \item $q_{n+2}h_{0}h_{n}h_{n+3}^2$
    \end{enumerate}
    For $n\ge 4$, classes above are distinct and nontrivial. When $n=3$, the terms $(4)$ and $(5)$ are trivial.
\end{lem}
\begin{proof}
    A class in degree $(4,1,2^{n+1}\cdot 21+1,*)$ must be of the form $q_i\cdot a$, where 
    \[  a\in \Ext_\P^{4, 2^{n+1}\cdot 21 - 2^{i+1} + 2} \]

    Case 1: $a$ is indecomposable. Based on the structure of the Adams $4$-line from \cite{LinExt} and a careful check of degrees, the only possibility for $q_i\cdot a$ is $q_0 e_n$ which is the case $(1)$ listed above.
    
    
    Case 2: $a$ is decomposable. Then the element $a$ must be of the form $c_a h_b$ or $h_a h_b h_c h_d$. 

    In the first sub-case, computing the $t$-degree shows:
    \[ 2^{a+4} + 2^{a+2} + 2^{a+1} + 2^{b+1} + 2^{i+1} = 2^{n+5} + 2^{n+3} + 2^{n+1} + 2 \]
    which is impossible by a careful analysis.

    In the second sub-case, computing the $t$-degree shows:
    \[ 2^{a+1} + 2^{b+1} + 2^{c+1} + 2^{d+1} + 2^{i+1} = 2^{n+5} + 2^{n+3} + 2^{n+1} + 2 \]
    so the set (with multiplicity) $\{ a,b,c,d,i \}$ should be:
    \begin{itemize}
        \item $\{0,n-1,n-1,n+2,n+4\}$
        \item $\{0,n,n+1,n+1,n+4\}$
        \item $\{0,n,n+2,n+3,n+3\}$
    \end{itemize}
    which gives the last $6$ classes. (The other $6$ classes are $0$, due to the relation $h_i h_{i+1} = 0$.)
\end{proof}

Now we are able to show:
\begin{thm}
    \label{surviveBXSS}
    For $n\ge 3$, the element $q_0\cdot x_{n-1}\in \Ext_\P(\FF,\FF)\otimes \FF[q_0,q_1,\cdots]$ survives to the $E_\infty$-page of the BXSS.
\end{thm}

\begin{proof}
    We will compute the differentials of the terms in the lemma above, and show that none of them kills $q_0 x_{n-1}$. We begin by noticing that, using \Cref{lem permanent cycle BXSS}, the cases $(1),(2),(3),(6),(7)$ can be written as a product of two permanent cycles, hence permanent. 

    For $n\ge 4$, we need to deal with $(4)$ and (5): Using \Cref{lem diff algBXSS}, the case $(4)$ admits nontrivial $d_1$:
    \[ d_1(q_{n+2}h_0h_{n-1}^2h_{n+4}) = q_{n+1}h_0h_{n-1}^2h_{n+1}h_{n+4} = q_{n+1}h_0h_n^3h_{n+4} \neq 0 \]
    and the case $(5)$ admits nontrivial $d_2$:
    \[ d_2(q_{n+4}h_0h_{n-1}^2h_{n+2}) = q_{n+2}\langle h_{n+2}, h_{n+3}, h_0h_{n-1}^2h_{n+2} \rangle = q_{n+2} h_0 h_{n-1}^2 h_{n+3}^2\neq 0 \]

\end{proof}

\section{The Cartan-Eilenberg Spectral Sequence}
\label{sec:CESS}

The element $q_0 x_{n-1}$ in the $E_1$-page of the CESS is of degree $(s,k,t) = (5,1,2^{n+1}\cdot 21 + 1)$. We will check that the potential sources of differentials killing $q_0 x_{n-1}$ are $0$.

\begin{lem}
\label{CESS filtration 3 die}
    For $n\ge 3$, $\Ext_{\P}^{2,2^{n+1}\cdot 21+1}(\FF,\Ext_{\mathcal Q}^{3}(\FF,\FF))\cong 0$.
\end{lem}
\begin{proof}
    Let's consider elements in the $E_1$-page of the BXSS of degree $(2,3,2^{n+1}\cdot 21 +1,*)$: Such an element must be of the form $q_aq_bq_ch_ih_j$. Computing the $t$-degree, we get:
    \[ 2^{a+1} + 2^{b+1} + 2^{c+1} + 2^{i+1} + 2^{j+1} = 2^{n+5} + 2^{n+3} + 2^{n+1} + 2^2 \]
    Therefore, the set (with multiplicity) $\{ a,b,c,i,j \}$ must be one of the following:
    \begin{itemize}
        \item $\{0, 0, n, n+2, n+4\}$
        \item $\{1, n-1, n-1, n+2, n+4\}$
        \item $\{1, n, n+1, n+1, n+4\}$
        \item $\{1, n, n+2, n+3, n+3\}$
    \end{itemize}
    Therefore, the elements listed by degrees are: 
    
    \begin{multicols}{4}
    \begin{enumerate}
    \item $q_0^2 q_n h_{n+2} h_{n+4}$
    \item $q_0^2 q_{n+2} h_n h_{n+4}$
    \item $q_0^2 q_{n+4} h_n h_{n+2}$
    \item $q_0 q_n q_{n+4} h_0 h_{n+2}$
    \item $q_0 q_n q_{n+2} h_0 h_{n+4}$
    \item $q_0 q_{n+2} q_{n+4} h_0 h_n$
    \item $q_n q_{n+2} q_{n+4} h_0^2$
    \item $q_1 q_{n-1}^2 h_{n+2} h_{n+4}$
    \item $q_1 q_{n-1} q_{n+2} h_{n-1} h_{n+4}$
    \item $q_{n-1}^2 q_{n+2} h_1 h_{n+4}$
    \item $q_1 q_{n-1} q_{n+4} h_{n-1} h_{n+2}$
    \item $q_{n-1}^2 q_{n+4} h_1 h_{n+2}$
    \item $q_1 q_{n+2} q_{n+4} h_{n-1}^2$
    \item $q_{n-1} q_{n+2} q_{n+4} h_1 h_{n-1}$
    \item $q_1 q_n q_{n+1} h_{n+1} h_{n+4}$
    \item $q_1 q_{n+1}^2 h_n h_{n+4}$
    \item $q_n q_{n+1}^2 h_1 h_{n+4}$
    \item $q_1 q_n q_{n+4} h_{n+1}^2$
    \item $q_n q_{n+1} q_{n+4} h_1 h_{n+1}$
    \item $q_{n+1}^2 q_{n+4} h_1 h_n$
    \item $q_1 q_n q_{n+2} h_{n+3}^2$
    \item $q_1 q_{n+2} q_{n+3} h_n h_{n+3}$
    \item $q_n q_{n+2} q_{n+3} h_1 h_{n+3}$
    \item $q_1 q_{n+3}^2 h_n h_{n+2}$
    \item $q_n q_{n+3}^2 h_1 h_{n+2}$
    \item $q_{n+2} q_{n+3}^2 h_1 h_n$
    \end{enumerate}
    \end{multicols}

    For $n\ge 4$, all of the elements above are nontrivial and distinct. 



    None of the elements survive to the CESS $E_1$-page. Here are the differentials listed by degrees: 
    
    (Those ``$\neq 0$'' holds for all $n\ge 4$. When $n=3$ we need some more computations.)

\begin{enumerate}
    \item  $d_1(q_0^2 q_n h_{n+2} h_{n+4}) = q_0^2 q_{n-1} h_{n-1} h_{n+2} h_{n+4} \neq 0$
    \item  $d_1(q_0^2 q_{n+2} h_n h_{n+4}) = q_0^2 q_{n+1} h_n h_{n+1} h_{n+4} = 0$ \\
    $d_2(q_0^2 q_{n+2} h_n h_{n+4}) = q_0^2 q_n \langle h_n, h_{n+1}, h_n \rangle h_{n+4} = q_0^2 q_n h_{n+1}^2 h_{n+4} \neq 0$
    \item  $d_1(q_0^2 q_{n+4} h_n h_{n+2}) = q_0^2 q_{n+3} h_n h_{n+2} h_{n+3} = 0 $  \\
    $d_2(q_0^2 q_{n+4} h_n h_{n+2}) = q_0^2 q_{n+2} \langle h_{n+2}, h_{n+3}, h_{n+2} \rangle h_n = q_0^2 q_{n+2} h_n h_{n+3}^2 \neq 0$
    \textcolor{blue}{
    \item  $d_1(q_0 q_n q_{n+4} h_0 h_{n+2}) = q_0 q_{n-1} q_{n+4} h_0 h_{n-1} h_{n+2} \neq 0$
    }
    \textcolor{blue}{
    \item  $d_1(q_0 q_n q_{n+2} h_0 h_{n+4}) = q_0 q_{n-1} q_{n+2} h_0 h_{n-1} h_{n+4} + q_0 q_n q_{n+1} h_0 h_{n+1} h_{n+4} \neq 0$
    }
    \textcolor{blue}{
    \item  $d_1(q_0q_{n+2}q_{n+4}h_0h_n) = q_0q_{n+2}q_{n+3}h_0h_nh_{n+3}\not=0$
    }
    
   \item  $d_1(q_n q_{n+2} q_{n+4} h_0^2) = q_{n-1} q_{n+2} q_{n+4} h_0^2 h_{n-1} + q_n q_{n+1} q_{n+4} h_0^2 h_{n+1} + q_n q_{n+2} q_{n+3} h_0^2 h_{n+3} \neq 0$
   \item  $d_1(q_1 q_{n-1}^2 h_{n+2} h_{n+4}) = q_0 q_{n-1}^2 h_0 h_{n+2} h_{n+4} \neq 0$
    \textcolor{blue}{
   \item  $d_1(q_1 q_{n-1} q_{n+2} h_{n-1} h_{n+4}) = q_0 q_{n-1} q_{n+2} h_0 h_{n-1} h_{n+4} + q_1 q_{n-1} q_{n+1} h_{n-1} h_{n+1} h_{n+4} \neq 0$
    }
   \item  $d_1(q_{n-1}^2 q_{n+2} h_1 h_{n+4}) = q_{n-1}^2 q_{n+1} h_1 h_{n+1} h_{n+4} \neq 0$
    \textcolor{blue}{
   \item  $d_1(q_1 q_{n-1} q_{n+4} h_{n-1} h_{n+2}) = q_0 q_{n-1} q_{n+4} h_0 h_{n-1} h_{n+2} \neq 0$
    }
   \item  $d_1(q_{n-1}^2 q_{n+4} h_1 h_{n+2}) = 0.$  \\
    $d_2(q_{n-1}^2 q_{n+4} h_1 h_{n+2}) = q_{n-2}^2 q_{n+4} h_1 h_{n-1} h_{n+2} + q_{n-1}^2 q_{n+2} h_1 h_{n+3}^2$
   \item  $d_1(q_1q_{n+2}q_{n+4}h_{n-1}^2) = q_0q_{n+2}q_{n+4}h_0h_{n-1}^2+q_1q_{n+1}q_{n+4}h_{n-1}^2h_{n+2}+q_1q_{n+2}q_{n+3}h_{n-1}^2h_{n+3}$
    \textcolor{blue}{
   \item  $d_1(q_{n-1} q_{n+2} q_{n+4} h_1 h_{n-1}) = q_{n-1} q_{n+1} q_{n+4} h_1 h_{n-1} h_{n+1} + q_{n-1} q_{n+2} q_{n+3} h_1 h_{n-1} h_{n+3} \neq 0$
    }
    \textcolor{blue}{
   \item  $d_1(q_1 q_n q_{n+1} h_{n+1} h_{n+4}) = q_0 q_n q_{n+1} h_0 h_{n+1} h_{n+4} + q_1 q_{n-1} q_{n+1} h_{n-1} h_{n+1} h_{n+4} \neq 0$
    }
   \item  $d_1(q_1 q_{n+1}^2 h_n h_{n+4}) = q_0 q_{n+1}^2 h_0 h_n h_{n+4} \neq 0$
   \item  $d_1(q_n q_{n+1}^2 h_1 h_{n+4}) = q_{n-1} q_{n+1}^2 h_1 h_{n-1} h_{n+4} \neq 0$
   \item  $d_1(q_1 q_n q_{n+4} h_{n+1}^2) = q_0 q_n q_{n+4} h_0 h_{n+1}^2 + q_1 q_n q_{n+3} h_{n+1}^2 h_{n+3} \neq 0$
    \textcolor{blue}{
   \item  $d_1(q_n q_{n+1} q_{n+4} h_1 h_{n+1}) = q_{n-1} q_{n+1} q_{n+4} h_1 h_{n-1} h_{n+1} + q_n q_{n+1} q_{n+3} h_1 h_{n+1} h_{n+3} \neq 0$
    }
   \item  $d_1(q_{n+1}^2 q_{n+4} h_1 h_n) = q_{n+1}^2 q_{n+3} h_1 h_n h_{n+3} \neq 0$
    \item  $d_1(q_1 q_n q_{n+2} h_{n+3}^2) = q_0 q_n q_{n+2} h_0 h_{n+3}^2 + q_1 q_{n-1} q_{n+2} h_{n-1} h_{n+3}^2 \neq 0$
    \textcolor{blue}{
   \item  $d_1(q_1 q_{n+2} q_{n+3} h_n h_{n+3}) = q_0 q_{n+2} q_{n+3} h_0 h_n h_{n+3} \neq 0$
    }
    \textcolor{blue}{
   \item  $d_1(q_n q_{n+2} q_{n+3} h_1 h_{n+3}) = q_{n-1} q_{n+2} q_{n+3} h_1 h_{n-1} h_{n+3} + q_n q_{n+1} q_{n+3} h_1 h_{n+1} h_{n+3} \neq 0$
    }
   \item  $d_1(q_1 q_{n+3}^2 h_n h_{n+2}) = q_0q_{n+3}^2 h_0 h_n h_{n+2} \neq 0$
   \item  $d_1(q_n q_{n+3}^2 h_1 h_{n+2}) = q_{n-1} q_{n+3}^2 h_1 h_{n-1} h_{n+2} \neq 0$
   \item  $d_1(q_{n+2} q_{n+3}^2 h_1 h_n) = 0, $  \\
    $d_2(q_{n+2} q_{n+3}^2 h_1 h_n) = q_n q_{n+3}^2 h_1 h_{n+1}^2 + q_{n+2}^3 h_1 h_n h_{n+3} \neq 0$
\end{enumerate}

The targets of differentials in \textcolor{blue}{blue} are not linearly independent. $(14)+(19)+(23)$, $(6)+(22)$, $(4)+(11)$, and $(5)+(9)+(15)$ has $d_1 = 0$. To fix this, we need four more differentials:
\begin{itemize}
    \item 
    $q_0 q_n q_{n+2} h_0 h_{n+4} + q_1 q_{n-1} q_{n+2}h_{n-1} h_{n+4} + q_1 q_n q_{n+1} h_{n+1} h_{n+4} = d_1( q_1 q_n q_{n+2} h_{n+4} )$
    \item 
    $q_0 q_n q_{n+4} h_0 h_{n+2} + q_1 q_{n-1} q_{n+4} h_{n-1} h_{n+2} = d_1( q_1 q_n q_{n+4} h_{n+2} )$
    \item 
    $q_0 q_{n+2} q_{n+4} h_0 h_n + q_1 q_{n+2} q_{n+3} h_n h_{n+3} = d_1( q_1 q_{n+2} q_{n+4} h_n )$
    \item 
    $q_{n-1} q_{n+2} q_{n+4} h_1 h_{n-1} + q_n q_{n+1} q_{n+4} h_1 h_{n+1} + q_n q_{n+2} q_{n+3} h_1 h_{n+3} = d_1( q_n q_{n+2} q_{n+4} h_1 )$
\end{itemize}

\begin{figure}[H]
\begin{center}
  {\renewcommand{\arraystretch}{1.1}
    \begin{tabular}{|wc{1cm}||wc{2.5cm}|p{5cm}|wc{4cm}|} \hline
      & $s=1$ & \multicolumn{1}{p{5cm}|}{\centering $s=2$} & $s=3$ \\\hline\hline
    $3n+8$  & & $(26)\ q_{n+2} q_{n+3}^2 h_1 h_n$ \tikzmark{26} &  \\\hline
    $3n+7$  & & & \\\hline          
    $3n+6$  & &$(7)\ q_n q_{n+2} q_{n+4} h_0^2$ \tikzmark{7} &\tikzmark{26t2}$q_{n+2}^3 h_1 h_n h_{n+3}$    \\
            & &$(20)\ q_{n+1}^2 q_{n+4} h_1 h_n$ \tikzmark{20}& \tikzmark{26t1}$q_n q_{n+3}^2 h_1 h_{n+1}^2$ \\
            &$q_n q_{n+2} q_{n+4} h_1$\tikzmark{g4} &$(25)\ q_n q_{n+3}^2 h_1 h_{n+2}$\tikzmark{25}&   \\\hline
    $3n+5$  & & &\tikzmark{7t1}$q_{n-1} q_{n+2} q_{n+4} h_0^2 h_{n-1}$  \\
            & & &\tikzmark{7t2} $q_n q_{n+1} q_{n+4} h_0^2 h_{n+1}$  \\
            & &\tikzmark{g4t1}$(14)\ q_{n-1} q_{n+2} q_{n+4} h_1 h_{n-1}\tikzmark{14}$ &\tikzmark{7t3} $q_n q_{n+2} q_{n+3} h_0^2 h_{n+3}$  \\
            & &\tikzmark{g4t2}$(19)\ q_n q_{n+1} q_{n+4} h_1 h_{n+1}$\tikzmark{19}  &\tikzmark{20t1} $q_{n+1}^2 q_{n+3} h_1 h_n h_{n+3}$ \\
            & &\tikzmark{g4t3}$(23)\ q_n q_{n+2} q_{n+3} h_1 h_{n+3}$\tikzmark{23} &\tikzmark{25t1} $q_{n-1} q_{n+3}^2 h_1 h_{n-1} h_{n+2}$
            \\\hline
    $3n+4$  & & &\tikzmark{14t1}\tikzmark{19t1}$q_{n-1} q_{n+1} q_{n+4} h_1 h_{n-1} h_{n+1}$  \\
            & & &\tikzmark{14t2}\tikzmark{23t1}$q_{n-1} q_{n+2} q_{n+3} h_1 h_{n-1} h_{n+3}$ \\
            & & &\tikzmark{19t2}\tikzmark{23t2}$q_n q_{n+1} q_{n+3} h_1 h_{n+1} h_{n+3}$
    \\\hline
    $3n+3$  & & &  \\\hline
    $3n+2$  & & $(12)\ q_{n-1}^2 q_{n+4} h_1 h_{n+2}$\tikzmark{12} &   \\
            & & $(17)\ q_n q_{n+1}^2 h_1 h_{n+4}$\tikzmark{17} & \\\hline
    $3n+1$  & & &\tikzmark{17t1}$q_{n-1} q_{n+1}^2 h_1 h_{n-1} h_{n+4}$  \\\hline
    
    $3n$    & & & \tikzmark{12t1}$q_{n-2}^2 q_{n+4} h_1 h_{n-1} h_{n+2}$ \\
            & &$(10)\ q_{n-1}^2 q_{n+2} h_1 h_{n+4}$\tikzmark{10}&\tikzmark{12t2}$q_{n-1}^2 q_{n+2} h_1 h_{n+3}^2$ \\\hline

    $3n-1$  & & & \tikzmark{10t1}$q_{n-1}^2 q_{n+1} h_1 h_{n+1} h_{n+4}$ \\\hline
    $\cdots$ & & & \\\hline
    $2n+7$  &$q_1 q_{n+2} q_{n+4} h_n$\tikzmark{g3} &$(13)\ q_1 q_{n+2} q_{n+4} h_{n-1}^2$\tikzmark{13} &   \\
            & &$(24)\ q_1 q_{n+3}^2 h_n h_{n+2}$\tikzmark{24} &  \\\hline
    $2n+6$  & & &\tikzmark{13t1}$q_0q_{n+2}q_{n+4}h_0h_{n-1}^2$ \\
    & & &\tikzmark{13t2}$q_1q_{n+1}q_{n+4}h_{n-1}^2h_{n+2}$\\
    & &\tikzmark{g3t1}$(6)\ q_0 q_{n+2} q_{n+4} h_0 h_n$\tikzmark{6} &\tikzmark{13t3}$q_1q_{n+2}q_{n+3}h_{n-1}^2h_{n+3}$ \\
            & &\tikzmark{g3t2}$(22)\ q_1 q_{n+2} q_{n+3} h_n h_{n+3}$\tikzmark{22} & \tikzmark{24t1}$q_0 q_{n+3}^2 h_0 h_n h_{n+2}$  \\\hline
    $2n+5$  &$q_1 q_n q_{n+4} h_{n+2}$\tikzmark{g2} &$(18)\ q_1 q_n q_{n+4} h_{n+1}^2$\tikzmark{18} & \tikzmark{6t1}\tikzmark{22t1}$q_0q_{n+2}q_{n+3}h_0h_nh_{n+3}$   \\\hline
    $2n+4$  & &\tikzmark{g2t1}$(4)\ q_0 q_n q_{n+4} h_0 h_{n+2}$\tikzmark{4} &\tikzmark{18t1}$q_0 q_n q_{n+4} h_0 h_{n+1}^2$  \\
            & &\tikzmark{g2t2}$(11)\ q_1 q_{n-1} q_{n+4} h_{n-1} h_{n+2}$\tikzmark{11} &\tikzmark{18t2}$q_1 q_n q_{n+3} h_{n+1}^2 h_{n+3}$  \\\hline
    $2n+3$  &$q_1 q_n q_{n+2} h_{n+4}$\tikzmark{g1} &$(16)\ q_1 q_{n+1}^2 h_n h_{n+4}$\tikzmark{16} &\tikzmark{4t1}\tikzmark{11t1}$q_0 q_{n-1} q_{n+4} h_0 h_{n-1} h_{n+2}$ \\
            & &$(21)\ q_1 q_n q_{n+2} h_{n+3}^2$\tikzmark{21} & \\\hline
    $2n+2$  & &\tikzmark{g1t1}$(5)\ q_0 q_n q_{n+2} h_0 h_{n+4}$\tikzmark{5} &\tikzmark{16t1}$q_0 q_{n+1}^2 h_0 h_n h_{n+4}$  \\
            & &\tikzmark{g1t2}$(9)\ q_1 q_{n-1} q_{n+2} h_{n-1} h_{n+4}$\tikzmark{9}  &\tikzmark{21t1}$q_0 q_n q_{n+2} h_0 h_{n+3}^2$ \\
            & &\tikzmark{g1t3}$(15)\ q_1 q_n q_{n+1} h_{n+1} h_{n+4}$\tikzmark{15} &\tikzmark{21t2}$q_1 q_{n-1} q_{n+2} h_{n-1} h_{n+3}^2$  \\\hline
    $2n+1$  & & &\tikzmark{5t1}\tikzmark{9t1}$q_0 q_{n-1} q_{n+2} h_0 h_{n-1} h_{n+4}$ \\
            & & &\tikzmark{5t2}\tikzmark{15t1}$q_0 q_n q_{n+1} h_0 h_{n+1} h_{n+4}$ \\
            & & & \tikzmark{9t2}\tikzmark{15t2}$q_1 q_{n-1} q_{n+1} h_{n-1} h_{n+1} h_{n+4}$ \\\hline
    $2n$    & & & \\\hline
    $2n-1$  & & $(8)\ q_1 q_{n-1}^2 h_{n+2} h_{n+4}$\tikzmark{8} & \\\hline
    $2n-2$  & & &\tikzmark{8t1}$q_0 q_{n-1}^2 h_0 h_{n+2} h_{n+4}$ \\\hline
    \end{tabular}

    \begin{tikzpicture}[overlay, remember picture, shorten >=.5pt, shorten <=.5pt, transform canvas={yshift=.25\baselineskip}]

    \draw [->] ({pic cs:26}) [blue] to ({pic cs:26t1}); 
    \draw [->] ({pic cs:26}) [blue] to ({pic cs:26t2});
    
    \draw [->] ({pic cs:7}) [red] to ({pic cs:7t1});
    \draw [->] ({pic cs:7}) [red] to ({pic cs:7t2});
    \draw [->] ({pic cs:7}) [red] to ({pic cs:7t3});
   
    \draw [->] ({pic cs:20}) [red] to ({pic cs:20t1});
    
    \draw [->] ({pic cs:25}) [red] to ({pic cs:25t1});
    
    \draw [->] ({pic cs:14}) [red] to ({pic cs:14t1});
    \draw [->] ({pic cs:14}) [red] to ({pic cs:14t2});

    \draw [->] ({pic cs:19}) [red] to ({pic cs:19t1});
    \draw [->] ({pic cs:19}) [red] to ({pic cs:19t2});

    \draw [->] ({pic cs:23}) [red] to ({pic cs:23t1});
    \draw [->] ({pic cs:23}) [red,bend right=15] to ({pic cs:23t2});

    \draw [->] ({pic cs:12}) [blue] to ({pic cs:12t1});
    \draw [->] ({pic cs:12}) [blue,bend right=20] to ({pic cs:12t2});

    \draw [->] ({pic cs:17}) [red] to ({pic cs:17t1});

    \draw [->] ({pic cs:10}) [red] to ({pic cs:10t1});

    \draw [->] ({pic cs:13}) [red] to ({pic cs:13t1});
    \draw [->] ({pic cs:13}) [red] to ({pic cs:13t2});
    \draw [->] ({pic cs:13}) [red] to ({pic cs:13t3});

    \draw [->] ({pic cs:24}) [red] to ({pic cs:24t1});

    \draw [->] ({pic cs:6}) [red] to ({pic cs:6t1});

    \draw [->] ({pic cs:22}) [red] to ({pic cs:22t1});

    \draw [->] ({pic cs:18}) [red] to ({pic cs:18t1});
    \draw [->] ({pic cs:18}) [red] to ({pic cs:18t2});

    \draw [->] ({pic cs:4}) [red,bend left=20] to ({pic cs:4t1});

    \draw [->] ({pic cs:11}) [red] to ({pic cs:11t1});

    \draw [->] ({pic cs:16}) [red] to ({pic cs:16t1});

    \draw [->] ({pic cs:21}) [red] to ({pic cs:21t1});
    \draw [->] ({pic cs:21}) [red] to ({pic cs:21t2});

    \draw [->] ({pic cs:5}) [red,bend left=20] to ({pic cs:5t1});
    \draw [->] ({pic cs:5}) [red,bend left=10] to ({pic cs:5t2});
    
    \draw [->] ({pic cs:9}) [red,bend right=15] to ({pic cs:9t1});
    \draw [->] ({pic cs:9}) [red] to ({pic cs:9t2});
    
    \draw [->] ({pic cs:15}) [red] to ({pic cs:15t1});
    \draw [->] ({pic cs:15}) [red,bend right=20] to ({pic cs:15t2});

    \draw [->] ({pic cs:8}) [red] to ({pic cs:8t1});




    \draw [->] ({pic cs:g1}) [red] to ({pic cs:g1t1});
    \draw [->] ({pic cs:g1}) [red, bend right=10] to ({pic cs:g1t2});
    \draw [->] ({pic cs:g1}) [red,bend right=20] to ({pic cs:g1t3});
    
    \draw [->] ({pic cs:g2}) [red] to ({pic cs:g2t1});
    \draw [->] ({pic cs:g2}) [red] to ({pic cs:g2t2});

    \draw [->] ({pic cs:g3}) [red] to ({pic cs:g3t1});
    \draw [->] ({pic cs:g3}) [red,bend right=10] to ({pic cs:g3t2});

    \draw [->] ({pic cs:g4}) [red] to ({pic cs:g4t1});
    \draw [->] ({pic cs:g4}) [red,bend right=10] to ({pic cs:g4t2});
    \draw [->] ({pic cs:g4}) [red,bend right=20] to ({pic cs:g4t3});

    \end{tikzpicture}

}\end{center}
\end{figure}

    \begin{figure}[H]
\begin{center}
  {\renewcommand{\arraystretch}{1.1}
    \begin{tabular}{|wc{1cm}||wc{2.5cm}|p{5cm}|wc{4cm}|} \hline
     & $s=1$ & \multicolumn{1}{p{5cm}|}{\centering $s=2$} &  $s=3$ \\\hline\hline
    $n+4$   & & $(3)\ q_0^2 q_{n+4} h_n h_{n+2}$\tikzmark{3} & \\\hline
    $n+3$   & & & \\\hline
    $n+2$   & & $(2)\ q_0^2 q_{n+2} h_n h_{n+4}$\tikzmark{2} &\tikzmark{3t1}$q_0^2 q_{n+2} h_n h_{n+3}^2$ \\\hline
    $n+1$   & & & \\\hline
    $n$     & & $(1)\ q_0^2 q_n h_{n+2} h_{n+4}$\tikzmark{1} &\tikzmark{2t1}$q_0^2 q_n h_{n+1}^2 h_{n+4}$ \\\hline
    $n-1$   & & &\tikzmark{1t1}$q_0^2 q_{n-1} h_{n-1} h_{n+2} h_{n+4}$ \\\hline
  \end{tabular}

  \begin{tikzpicture}[overlay, remember picture, shorten >=.5pt, shorten <=.5pt, transform canvas={yshift=.25\baselineskip}]

    \draw [->] ({pic cs:3}) [blue] to ({pic cs:3t1});

    \draw [->] ({pic cs:2}) [blue] to ({pic cs:2t1});

    \draw [->] ({pic cs:1}) [red] to ({pic cs:1t1});

    


  \end{tikzpicture}}
\end{center}

\caption{The BXSS in degree $(s,k,t) = (2,3,2^{n+1}\cdot 21+1)$ for Lemma \cref{CESS filtration 3 die}, where $n\geq 4$.
The vertical axis is the BX filtration and the horizontal axis is the $s$-degree. Red arrows indicate $d_1$-differentials, and blue arrows indicate $d_2$-differentials.}
\label{fig:CESS filtration 3 die}
\end{figure} 
When $n$ is small, the BX filtration on the left-most column may repeat. But for $n\ge 4$, all the terms appearing in the table above are distinct and nontrivial.\\ 

In the case $n=3$, most terms and differentials remain the same except for the following:
\begin{itemize}
    \item $(25)\ d_2(q_3q_6^2h_1h_5) = d_2(q_3h_1\cdot q_6^2h_5 )=q_1h_2^2\cdot q_6^2h_5+q_3h_1\cdot q_5^2h_6h_5=q_1q_6^2h_2^2h_5\neq 0$
    \item $(14)\ q_2q_5q_7h_1h_2=0$
    \item $(19)\ d_1(q_3q_4q_7h_1h_4) = q_3q_4q_6h_1h_4h_6 \not=0$
    \item $(23)\ d_1(q_3q_5q_6h_1h_6) = q_3q_4q_6h_1h_4h_6\not=0$
    \item $(17)\ d_2(q_3q_{4}^2h_1h_{7}) = d_2(q_3h_1\cdot q_4^2h_7)=q_1h_2^2\cdot q_4^2h_7+q_3h_1\cdot q_3^2h_4h_7\not=0$
    \item $(12)\ d_2(q_2^2q_7h_1h_5) = q_2^2q_5h_1h_6^2 \not=0$
    \item $(13)\ d_1(q_1q_5q_7h_2^2) = q_1q_4q_7h_2^2h_5+q_1q_5q_6h_2^2h_6\not=0$
\end{itemize}

\begin{figure}[H]
\begin{center}
  {\renewcommand{\arraystretch}{1.1}
    \begin{tabular}{|c||c|p{5cm}|c|}\hline
      & $s=1$ & \multicolumn{1}{p{5cm}|}{\centering $s=2$} & $s=3$ \\\hline\hline
    $17$  & & $(26)\ q_{5} q_{6}^2 h_1 h_3$ \tikzmark{n26} &  \\\hline
    $16$  & & & \\\hline          
    $15$  & &$(7)\ q_3 q_{5} q_{7} h_0^2$ \tikzmark{n7} &\tikzmark{n26t2}$q_{5}^3 h_1 h_3 h_{6}$    \\
            & & $(20)\ q_{4}^2 q_{7} h_1 h_3$ \tikzmark{n20}& \tikzmark{n26t1}$q_3 q_{6}^2 h_1 h_{4}^2$ \\
            &$q_3 q_{5} q_{7} h_1$\tikzmark{ng4} &$(25)\ q_3 q_{n+3}^2 h_1 h_{5}$\tikzmark{n25}&   \\\hline
    $14$  & & &\tikzmark{n7t1}$q_{2} q_{5} q_{7} h_0^2 h_{2}$  \\
            & & &\tikzmark{n7t2} $q_3 q_{4} q_{7} h_0^2 h_{4}$  \\
            & & &\tikzmark{n7t3} $q_3 q_{5} q_{6} h_0^2 h_{6}$  \\
            & &\tikzmark{ng4t2}$(19)\ q_3 q_{4} q_{7} h_1 h_{4}$\tikzmark{n19}  &\tikzmark{n20t1} $q_{4}^2 q_{6} h_1 h_3 h_{6}$ \\
            & &\tikzmark{ng4t3}$(23)\ q_3 q_{5} q_{6} h_1 h_{6}$\tikzmark{n23} &
            \\\hline
    $13$  &$q_1 q_{5} q_{7} h_3$\tikzmark{ng3} &$(13)\ q_1 q_{5} q_{7} h_{2}^2$\tikzmark{n13} & \tikzmark{n25t1}$q_1q_6^2h_2^2h_5$   \\
            & &$(24)\ q_1 q_{6}^2 h_3 h_{5}$\tikzmark{n24} & \tikzmark{n19t2}\tikzmark{n23t2}$q_3 q_{4} q_{6} h_1 h_{4} h_{6}$
    \\\hline
    $12$  & & &\tikzmark{n13t2}$q_1q_{4}q_{7}h_{2}^2h_{5}$\\
            & &\tikzmark{ng3t1}$(6)\ q_0 q_{5} q_{7} h_0 h_3$\tikzmark{n6} &\tikzmark{n13t3}$q_1q_{5}q_{6}h_{2}^2h_{6}$ \\
            & &\tikzmark{ng3t2}$(22)\ q_1 q_{5} q_{6} h_3 h_{6}$\tikzmark{n22} & \tikzmark{n24t1}$q_0 q_{6}^2 h_0 h_3 h_{5}$  \\\hline
    $11$  
        & & $(12)\ q_{2}^2 q_{7} h_1 h_{5}$\tikzmark{n12} &   \\
            & & $(17)\ q_3 q_{4}^2 h_1 h_{7}$\tikzmark{n17} & \\
            &$q_1 q_3 q_{7} h_{5}$\tikzmark{ng2} &$(18)\ q_1 q_3 q_{7} h_{4}^2$\tikzmark{n18} & \tikzmark{n6t1}\tikzmark{n22t1}$q_0q_{5}q_{6}h_0h_3h_{6}$
            \\\hline
            
    $10$  & &\tikzmark{ng2t1}$(4)\ q_0 q_3 q_{7} h_0 h_{5}$\tikzmark{n4} &\tikzmark{n18t1}$q_0 q_3 q_{7} h_0 h_{4}^2$  \\
            & &\tikzmark{ng2t2}$(11)\ q_1 q_{2} q_{7} h_{2} h_{5}$\tikzmark{n11} &\tikzmark{n18t2}$q_1 q_3 q_{6} h_{4}^2 h_{6}$ \\\hline
    
    $9$   
            & &$(10)\ q_{2}^2 q_{5} h_1 h_{7}$\tikzmark{n10}&\tikzmark{n12t2}$q_{2}^2 q_{5} h_1 h_{6}^2$ \\
            &$q_1 q_3 q_{5} h_{7}$\tikzmark{ng1} &$(16)\ q_1 q_{4}^2 h_3 h_{7}$\tikzmark{n16} & \tikzmark{n17t1}$q_1q_4^2h_2^2h_7+q_3^3h_1h_7$\\
            & &$(21)\ q_1 q_3 q_{5} h_{6}^2$\tikzmark{n21} & \tikzmark{n4t1}\tikzmark{n11t1}$q_0 q_{2} q_{7} h_0 h_{2} h_{5}$
            \\\hline

    \end{tabular}
    \begin{tikzpicture}[overlay, remember picture, shorten >=.5pt, shorten <=.5pt, transform canvas={yshift=.25\baselineskip}]

    \draw [->] ({pic cs:n26}) [blue] to ({pic cs:n26t1}); 
    \draw [->] ({pic cs:n26}) [blue] to ({pic cs:n26t2});
    
    \draw [->] ({pic cs:n7}) [red] to ({pic cs:n7t1});
    \draw [->] ({pic cs:n7}) [red] to ({pic cs:n7t2});
    \draw [->] ({pic cs:n7}) [red] to ({pic cs:n7t3});
   
    \draw [->] ({pic cs:n20}) [red] to ({pic cs:n20t1});
    
    \draw [->] ({pic cs:n25}) [blue] to ({pic cs:n25t1});
    

    \draw [->] ({pic cs:n19}) [red] to ({pic cs:n19t2});

    \draw [->] ({pic cs:n23}) [red,bend right=15] to ({pic cs:n23t2});

    \draw [->] ({pic cs:n12}) [blue] to ({pic cs:n12t2});

    \draw [->] ({pic cs:n17}) [blue] to ({pic cs:n17t1});


    \draw [->] ({pic cs:n13}) [red] to ({pic cs:n13t2});
    \draw [->] ({pic cs:n13}) [red] to ({pic cs:n13t3});

    \draw [->] ({pic cs:n24}) [red] to ({pic cs:n24t1});

    \draw [->] ({pic cs:n6}) [red] to ({pic cs:n6t1});

    \draw [->] ({pic cs:n22}) [red] to ({pic cs:n22t1});

    \draw [->] ({pic cs:n18}) [red] to ({pic cs:n18t1});
    \draw [->] ({pic cs:n18}) [red] to ({pic cs:n18t2});

    \draw [->] ({pic cs:n4}) [red] to ({pic cs:n4t1});

    \draw [->] ({pic cs:n11}) [red] to ({pic cs:n11t1});



    
    





    
    \draw [->] ({pic cs:ng2}) [red] to ({pic cs:ng2t1});
    \draw [->] ({pic cs:ng2}) [red] to ({pic cs:ng2t2});

    \draw [->] ({pic cs:ng3}) [red] to ({pic cs:ng3t1});
    \draw [->] ({pic cs:ng3}) [red,bend right=10] to ({pic cs:ng3t2});

    \draw [->] ({pic cs:ng4}) [red,bend right=10] to ({pic cs:ng4t2});
    \draw [->] ({pic cs:ng4}) [red,bend right=20] to ({pic cs:ng4t3});

    \end{tikzpicture}
}\end{center}

\end{figure}

\begin{figure}[H]
\begin{center}
  {\renewcommand{\arraystretch}{1.1}
    \begin{tabular}{|c||c|p{5cm}|c|}

        \hline
      & $s=1$ & \multicolumn{1}{p{5cm}|}{\centering $s=2$} & $s=3$ \\\hline\hline

        $9$   
            & &$(10)\ q_{2}^2 q_{5} h_1 h_{7}$\tikzmark{n10}&$q_{2}^2 q_{5} h_1 h_{6}^2$ \\
            &$q_1 q_3 q_{5} h_{7}$\tikzmark{ng1} &$(16)\ q_1 q_{4}^2 h_3 h_{7}$\tikzmark{n16} & $q_1q_4^2h_2^2h_7+q_3^3h_1h_7$\\
            & &$(21)\ q_1 q_3 q_{5} h_{6}^2$\tikzmark{n21} & $q_0 q_{2} q_{7} h_0 h_{2} h_{5}$
            \\\hline

        $8$  & & & \tikzmark{n10t1}$q_{2}^2 q_{4} h_1 h_{4} h_{7}$ \\
    & &\tikzmark{ng1t1}$(5)\ q_0 q_3 q_{5} h_0 h_{7}$\tikzmark{n5} &\tikzmark{n16t1}$q_0 q_{4}^2 h_0 h_3 h_{7}$  \\
            & &\tikzmark{ng1t2}$(9)\ q_1 q_{2} q_{5} h_{2} h_{7}$\tikzmark{n9}  &\tikzmark{n21t1}$q_0 q_3 q_{5} h_0 h_{6}^2$ \\
            & &\tikzmark{ng1t3}$(15)\ q_1 q_3 q_{4} h_{4} h_{7}$\tikzmark{n15} &\tikzmark{n21t2}$q_1 q_{2} q_{5} h_{2} h_{6}^2$\\\hline

    $7$  & & &\tikzmark{n5t1}\tikzmark{n9t1}$q_0 q_{2} q_{5} h_0 h_{2} h_{7}$ \\
            & & &\tikzmark{n5t2}\tikzmark{n15t1}$q_0 q_3 q_{4} h_0 h_{4} h_{7}$ \\
            & &$(3)\ q_0^2 q_{7} h_3 h_{5}$\tikzmark{n3} & \tikzmark{n9t2}\tikzmark{n15t2}$q_1 q_{2} q_{4} h_{2} h_{4} h_{7}$ \\\hline
    $6$    & & & \\\hline
    $5$  & & $(8)\ q_1 q_{2}^2 h_{5} h_{7}$\tikzmark{n8} & \\
        & & $(2)\ q_0^2 q_{5} h_3 h_{7}$\tikzmark{n2} &\tikzmark{n3t1}$q_0^2 q_{5} h_3 h_{6}^2$\\\hline
    $4$  & & &\tikzmark{n8t1}$q_0 q_{2}^2 h_0 h_{5} h_{7}$ \\\hline
    $3$     & & $(1)\ q_0^2 q_3 h_{5} h_{7}$\tikzmark{n1} &\tikzmark{n2t1}$q_0^2 q_3 h_{4}^2 h_{7}$ \\\hline
    $2$   & & &\tikzmark{n1t1}$q_0^2 q_{2} h_{2} h_{5} h_{7}$ \\\hline
    
    \end{tabular}

    \begin{tikzpicture}[overlay, remember picture, shorten >=.5pt, shorten <=.5pt, transform canvas={yshift=.25\baselineskip}]

    
   
    
    





    \draw [->] ({pic cs:n10}) [red] to ({pic cs:n10t1});








    \draw [->] ({pic cs:n16}) [red] to ({pic cs:n16t1});

    \draw [->] ({pic cs:n21}) [red] to ({pic cs:n21t1});
    \draw [->] ({pic cs:n21}) [red] to ({pic cs:n21t2});

    \draw [->] ({pic cs:n5}) [red,bend left=20] to ({pic cs:n5t1});
    \draw [->] ({pic cs:n5}) [red,bend left=10] to ({pic cs:n5t2});
    
    \draw [->] ({pic cs:n9}) [red,bend right=15] to ({pic cs:n9t1});
    \draw [->] ({pic cs:n9}) [red] to ({pic cs:n9t2});
    
    \draw [->] ({pic cs:n15}) [red] to ({pic cs:n15t1});
    \draw [->] ({pic cs:n15}) [red,bend right=20] to ({pic cs:n15t2});

    \draw [->] ({pic cs:n8}) [red] to ({pic cs:n8t1});

    \draw [->] ({pic cs:n3}) [blue] to ({pic cs:n3t1});

    \draw [->] ({pic cs:n2}) [blue] to ({pic cs:n2t1});

    \draw [->] ({pic cs:n1}) [red] to ({pic cs:n1t1});

    \draw [->] ({pic cs:ng1}) [red] to ({pic cs:ng1t1});
    \draw [->] ({pic cs:ng1}) [red, bend right=10] to ({pic cs:ng1t2});
    \draw [->] ({pic cs:ng1}) [red,bend right=20] to ({pic cs:ng1t3});
    



    \end{tikzpicture}

}\end{center}

\caption{The BXSS in degree $(s,k,t) = (2,3,2^{3+1}\cdot 21+1)$ for Lemma \cref{CESS filtration 3 die}.
The vertical axis is the BX-filtration and the horizontal axis is the $s$-degree. Red arrows indicate $d_1$-differentials, and blue arrows indicate $d_2$-differentials.}
\label{fig:CESS filtration 3 die n=3}
\end{figure}

Using Proposition~\ref{lem diff algBXSS} we can see that, all the nontrivial terms above appearing as the target of some $d_2$ supports trivial $d_1$. Furthermore, since they cannot be written as sums of terms divisible by any $q_kh_k$, they are not killed by $d_1$-differentials. Therefore, those elements are still nontrivial on the $E_2$-page of the BXSS.

Therefore, for $n\ge 3$, $\Ext_{\P}^{2,2^{n+1}\cdot 21+1}(\FF,\Ext_{\mathcal Q}^{3}(\FF,\FF))\cong 0$.

\end{proof}

\begin{lem}
    For $n\ge 3$, $\Ext_{\P}^{0,2^{n+1}\cdot 21+1}(\FF,\Ext_{\mathcal Q}^{5}(\FF,\FF))\cong 0$.
\end{lem}
\begin{proof}
    Let's consider elements in BXSS $E_1$ of degree $(0,5,2^{n+1}\cdot 21 +1,*)$: Such an element must be of the form $q_aq_bq_cq_dq_e$. Computing the $t$-degree, we get:
    \[ 2^{a+1} + 2^{b+1} + 2^{c+1} + 2^{d+1} + 2^{e+1} = 2^{n+5} + 2^{n+3} + 2^{n+1} + 2^2 + 2 \]
    This term must be $q_0q_1q_nq_{n+2}q_{n+4}$, which admits nontrivial differential
    \begin{align*}
        d_1(q_0q_1q_nq_{n+2}q_{n+4}) &=
    q_0^2q_nq_{n+2}q_{n+4}h_0+q_0q_1q_{n-1}q_{n+2}q_{n+4}h_{n-1}\\
    &+q_0q_1q_nq_{n+1}q_{n+4}h_{n+1}+q_0q_1q_nq_{n+2}q_{n+3}h_{n+3}\neq 0.
    \end{align*}
\end{proof}

Therefore we are able to show:

\begin{thm}
    \label{surviveCESS}
    For $n\ge 3$, the element $q_0\cdot x_{n-1}\in \Ext_\P(\FF, \Ext_{\mathcal Q}(\FF,\FF))$ survives to the $E_\infty$-page of the CESS.
\end{thm}

\begin{proof}
    The $E_1$-page of BXSS vanishes when $k+t$ is odd, so the same for $E_2$-page of CESS. Therefore, if some element kills $q_0 x_{n-1}$, then it must be of degree $(2,3,2^{n+1}\cdot 21 + 1)$ or $(0,5,2^{n+1}\cdot 21 + 1)$. From the lemmas above, $E_2^{2,3,2^{n+1}\cdot 21+1}$ and $E_2^{0,5,2^{n+1}\cdot 21+1}$ are trivial, so we conclude that $q_0\cdot x_{n-1}$ survives to the $E_\infty$-page.

\end{proof}

Combining Theorem~\ref{surviveBXSS} and Theorem~\ref{surviveCESS}, we can see that, for $n\ge 3$, the element
\[h_0\cdot x_n\in \Ext_{\A}^{6,2^{j+1}\cdot 21+1}(\FF,\FF)\]
is nontrivial, which is our main theorem, Theorem~\ref{mainthm}.

\bibliographystyle{alpha}
\bibliography{bibliography}

\end{document}